\documentclass[11pt]{amsart}
\usepackage{graphicx}
\usepackage{amsmath,amsxtra,amssymb,latexsym, amscd,amsthm}

\usepackage[colorlinks=true,citecolor=black,linkcolor=black,urlcolor=blue]{hyperref}

\newtheorem{theorem}{Theorem}[section]

\newtheorem{lemma}{Lemma}[section]
\newtheorem{proposition}{Proposition}[section]
\newtheorem{definition}{Definition}[section]

\newtheorem{remark}{Remark}
\numberwithin{equation}{section}
\numberwithin{figure}{section}
\newcommand{\M}{\operatorname{M}}

\newcommand{\U}{\operatorname{U}}
\newcommand{\D}{\operatorname{D}}
\newcommand{\F}{\operatorname{F}}
\newcommand{\Ba}{\operatorname{Bar}}
\begin{document}
\setlength{\baselineskip}{16truept}
\title{A generalization of Aztec diamond theorem, part I}
\author{ Tri Lai}
\address{Indiana University, Bloomington IN 47405}
\date{\today}

\begin{abstract}
We generalize Aztec diamond theorem (N. Elkies, G. Kuperberg, M. Larsen, and J. Propp
\emph{Alternating-sign matrices and domino tilings}, Journal  Algebraic Combinatoric, 1992) by showing that the numbers of tilings of a certain family of regions in the square lattice with
 southwest-to-northeast diagonals drawn in are given by powers of $2$. We present a proof for the generalization by using a bijection between domino tilings and non-intersecting lattice paths.
\end{abstract}
\maketitle

\section{Introduction}
Given a lattice in the plane, a (lattice) \textit{region} is a finite connected union of fundamental regions of that lattice. A \textit{tile} is the union of two fundamental regions sharing an edge. A \textit{tiling} of the region $R$ is a covering of $R$ by tiles so that there are no gaps or overlaps.

A \textit{perfect matching} of a graph $G$ is a collection of edges such that each vertex of $G$ is adjacent to precisely one edge in the collection. Denote by $\M(G)$ the number of perfect matchings of  graph $G$.
The tilings of a region $R$ can be naturally identified with the perfect matchings of its \textit{dual graph} (i.e., the graph whose vertices are the fundamental regions of $R$, and whose edges connect two fundamental regions precisely when they share an edge). In the view of this, we denote by $\M(R)$ the number of tilings of $R$.

The \textit{Aztec diamond region} of order $n$ is defined to be the union of all the unit squares with integral corners $(x,y)$ satisfying $|x|+|y|\leq n+1$ in the Cartesian coordinate system  (see Figure \ref{aztecdiamond} for an example of Aztec diamond region of order $4$).
 The number of tilings of an Aztec diamond region is given by the following theorem that was first proved by Elkies, Kuperberg, Larsen and Propp \cite{Elkies}.

\begin{theorem}[Aztec diamond theorem \cite{Elkies}]\label{Aztecthm}
The number of (domino) tilings of the Aztec diamond of order $n$ is $2^{n(n+1)/2}$.
\end{theorem}

\begin{figure}\centering
\includegraphics[width=7cm]{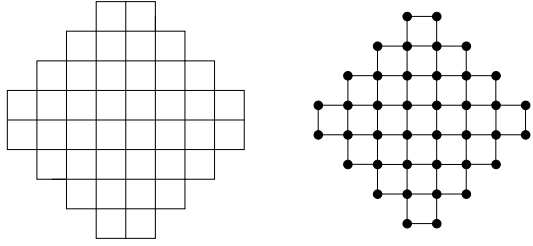}
\caption{The Aztec diamond region of order 4 (left) and its dual graph, the Aztec diamond graph of order 4 (right).}
\label{aztecdiamond}
\end{figure}

Douglas \cite{Doug} considered a certain family of regions in the square lattice with every second southwest-to-northeast diagonal drawn in (examples are shown in Figure \ref{douglas}). Precisely, the region of order $n$, denoted by $D(n)$,  has four vertices that are the vertices of a diamond of side-length $2n\sqrt{2}$.

\begin{theorem}[Douglas \cite{Doug}]\label{Dtheorem}
\begin{equation}
\M(D(n))=2^{2n(n+1)}.
\end{equation}
\end{theorem}

\begin{figure}
\centering
\begin{picture}(0,0)%
\includegraphics{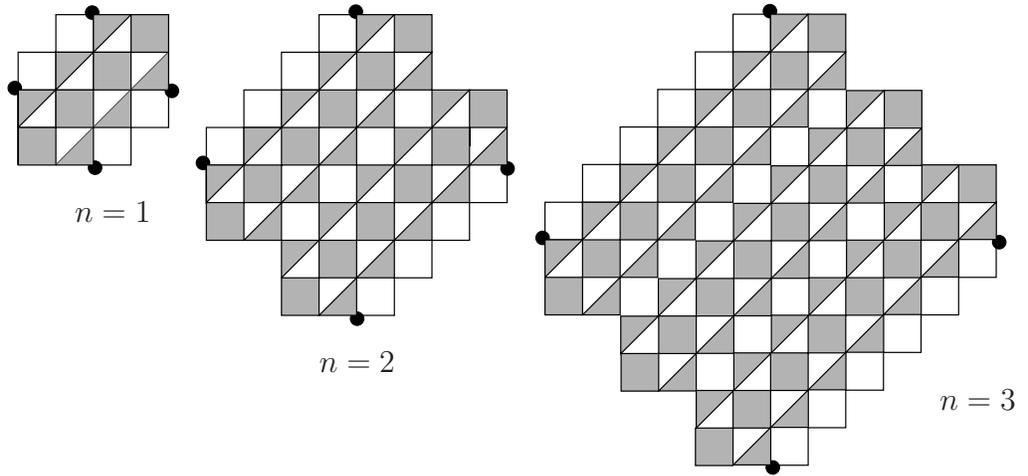}%
\end{picture}%
\setlength{\unitlength}{3947sp}%
\begingroup\makeatletter\ifx\SetFigFont\undefined%
\gdef\SetFigFont#1#2#3#4#5{%
  \reset@font\fontsize{#1}{#2pt}%
  \fontfamily{#3}\fontseries{#4}\fontshape{#5}%
  \selectfont}%
\fi\endgroup%
\begin{picture}(6283,2966)(402,-2292)
\put(2363,-1641){\makebox(0,0)[lb]{\smash{{\SetFigFont{12}{14.4}{\familydefault}{\mddefault}{\updefault}{$n=2$}%
}}}}
\put(6260,-1878){\makebox(0,0)[lb]{\smash{{\SetFigFont{12}{14.4}{\familydefault}{\mddefault}{\updefault}{$n=3$}%
}}}}
\put(828,-696){\makebox(0,0)[lb]{\smash{{\SetFigFont{12}{14.4}{\familydefault}{\mddefault}{\updefault}{$n=1$}%
}}}}
\end{picture}
\caption{The Douglas' regions of order $n=1$, $n=2$ and $n=3$.}
\label{douglas}
\end{figure}

The regions in the Douglas' theorem have the distances\footnote{The unit here is the distance between two consecutive lattice diagonals of the square lattice, i.e. $\sqrt{2}/2$.} between any two successive southwest-to-northeast diagonals drawn in are 2. Next, we consider general situation when the distances  between two successive drawn-in diagonals\footnote{ From now on, ``diagonal(s)" refers to ``southwest-to-northeast diagonal(s)"} are \textit{arbitrary}.

Consider the setup of drawn-in diagonals in the square lattice as follows. Let $\ell$ and $\ell'$ be two fixed lattice diagonals ($\ell$ and $\ell'$ are \textit{not} drawn-in diagonals), and assume that $k-1$ diagonals have been drawn in between $\ell$ and $\ell'$, with the distances between successive ones, starting from top, being $d_2,\dotsc,d_{k-2}$. The distance between $\ell$ and the top  drawn-in diagonal is $d_1$, and the distance between the bottom drawn-in diagonal and $\ell'$ is $d_k$.

Given a positive integer $a$, we define the region $D_a(d_1,\dotsc,d_k)$ as follows (see Figure \ref{DouglasSchroder} for an example).  Its southwestern and northeastern boundaries are defined in the next two paragraphs.

Color the resulting dissection of the square lattice black and white so that any two fundamental regions that share an edge have opposite colors, and assume that the fundamental regions passed through by $\ell$ are white (by definition $\ell$ and $\ell'$ pass through unit squares). Let $A$ be a lattice point on $\ell$. Start from $A$ and take unit steps south or east so that for each step the color of the fundamental region on the left is black. We arrive $\ell'$ at a lattice point $B$. The described path from $A$ to $B$ is the northeastern boundary of our region.

Let $D$ be the lattice point on $\ell$ that is $a$ unit square diagonals to the southwest of $A$ (i.e. $|AD|=a\sqrt{2}$).
The southwestern boundary is obtained from the northeastern  boundary by reflecting it about the perpendicular bisector of segment $AD$,
and reversing the directions of its unit steps (from south to north, and from east to west). Let $C$ be the reflection point of $B$ about the perpendicular bisector above, so $C$ is also on $\ell'$.

\begin{figure}\centering
\resizebox{!}{8cm}{
\begin{picture}(0,0)%
\includegraphics{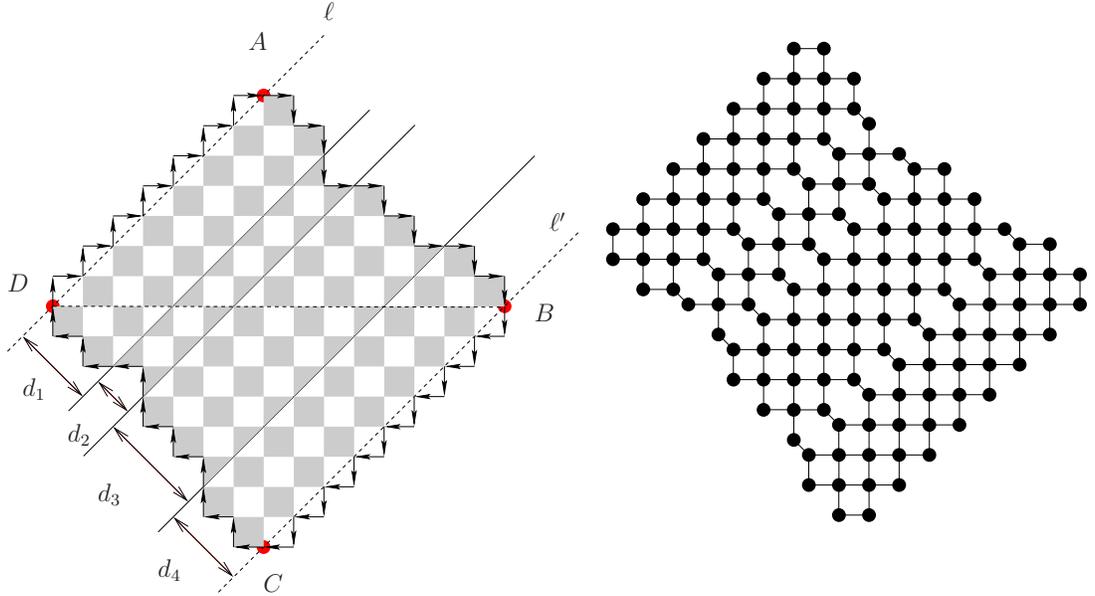}%
\end{picture}%
\setlength{\unitlength}{3947sp}%
\begingroup\makeatletter\ifx\SetFigFont\undefined%
\gdef\SetFigFont#1#2#3#4#5{%
  \reset@font\fontsize{#1}{#2pt}%
  \fontfamily{#3}\fontseries{#4}\fontshape{#5}%
  \selectfont}%
\fi\endgroup%
\begin{picture}(8479,4719)(1198,-4323)
\put(3103,-56){\makebox(0,0)[lb]{\smash{{\SetFigFont{12}{14.4}{\rmdefault}{\mddefault}{\updefault}{$A$}%
}}}}
\put(5347,-2182){\makebox(0,0)[lb]{\smash{{\SetFigFont{12}{14.4}{\rmdefault}{\mddefault}{\updefault}{$B$}%
}}}}
\put(3221,-4308){\makebox(0,0)[lb]{\smash{{\SetFigFont{12}{14.4}{\rmdefault}{\mddefault}{\updefault}{$C$}%
}}}}
\put(1213,-1946){\makebox(0,0)[lb]{\smash{{\SetFigFont{12}{14.4}{\rmdefault}{\mddefault}{\updefault}{$D$}%
}}}}
\put(3694,180){\makebox(0,0)[lb]{\smash{{\SetFigFont{12}{14.4}{\rmdefault}{\mddefault}{\updefault}{$\ell$}%
}}}}
\put(5465,-1473){\makebox(0,0)[lb]{\smash{{\SetFigFont{12}{14.4}{\rmdefault}{\mddefault}{\updefault}{$\ell'$}%
}}}}
\put(1332,-2772){\makebox(0,0)[lb]{\smash{{\SetFigFont{12}{14.4}{\rmdefault}{\mddefault}{\updefault}{$d_1$}%
}}}}
\put(1686,-3127){\makebox(0,0)[lb]{\smash{{\SetFigFont{12}{14.4}{\rmdefault}{\mddefault}{\updefault}{$d_2$}%
}}}}
\put(1922,-3599){\makebox(0,0)[lb]{\smash{{\SetFigFont{12}{14.4}{\rmdefault}{\mddefault}{\updefault}{$d_3$}%
}}}}
\put(2394,-4190){\makebox(0,0)[lb]{\smash{{\SetFigFont{12}{14.4}{\rmdefault}{\mddefault}{\updefault}{$d_4$}%
}}}}
\end{picture}}
\caption{The region $D_{7}(4,2,5,4)$ (left) and its dual graph (right).}
\label{DouglasSchroder}
\end{figure}

Connect $D$ and $A$ by a zigzag lattice path consisting of alternatively east and north steps, so that the unit squares passed through by $\ell$ are on the right of the zigzag path. Similarly, we connect $B$ and $C$ by a zigzag lattice path, so that the square cells  passed through by $\ell'$ are on the right. These two zigzag lattice paths are northwestern and southeastern boundaries, and they complete the boundary of the region $D_a(d_1,\dotsc,d_k)$. We call the resulting region a \textit{generalized Douglas region}.

\begin{remark}\label{rmk1}
(1) If the line $\ell'$ passes through black unit squares, then the region does not have a tiling (since we can not cover the black squares by disjoint tiles). Hereafter, \textit{ we assume that $\ell'$ passes through white unit square.}

(2) Since we only consider connected region, \textit{we also assume from now on that the southwestern and northeastern boundaries do not intersect each other.}
\end{remark}

We call the fundamental regions in a generalized Douglas region \textit{cells}.  Note that there are two kinds of cells, square and triangular. The latter in turn come in two orientations: they may point towards $\ell'$ or away from $\ell'$. We call them \textit{down-pointing triangles} or \textit{up-pointing triangles}, respectively.  A cell is said to be \textit{regular} if it is a black square or a black up-pointing triangle.

A \textit{row of cells} consists of all the triangular cells of a given color with bases resting on a fixed lattice diagonal, or  consists of all the square cells (of a given color) passed through by a fixed lattice diagonal.
Define the \textit{width} of our region to be the number of white squares in the bottom row of cells. One readily sees that the width of the region is exactly $|BC|/\sqrt{2}$, where $|BC|$ is the Euclidian distance between $B$ and $C$. The number of tilings of a generalized Douglas region is obtained by the theorem stated below.

\begin{theorem}\label{gendoug}
Assume that  $a,$ $d_1,$ $\dotsc,$ $d_k$ are positive integers, so that for which the generalized Douglas region $D_{a}(d_1,\dotsc,d_k)$ has the width  $w$, and has its western and eastern vertices (i.e. the vertices $B$ and $D$) on the same horizontal line. Then
\begin{equation}\label{gendougeq}
\M(D_a(d_1,\dotsc,d_k))=2^{\mathcal{C}-w(w+1)/2}
\end{equation}
where $\mathcal{C}$ is the number of regular cells in the region.
\end{theorem}

Let $k=1$ (i.e. there are \textit{no} dawn-in diagonals between $\ell$ and $\ell'$) and $a=d_1=n$, our generalized Douglas region, $D_a(d_1,\dotsc,d_k)$,
is exactly the Aztec diamond region of order $n$.
One readily sees that the region has the width $w=n$ and the number of regular cells $\mathcal{C}=n(n+1)$. This means that we can imply Aztec diamond theorem \ref{Aztecthm} from  Theorem \ref{gendoug}.

Moreover, one can get Douglas' theorem \ref{Dtheorem} from the Theorem \ref{gendoug} by setting   $k=2n\geq 2$, $d_1=d_k=1$, $a=k$, and $d_2=d_3=\dotsc=d_{k-1}=2$.
 Therefore, Theorem \ref{gendoug} can be view as a common multi-parameter generalization of Aztec diamond theorem and Douglas' theorem.

 For the sake of simplicity, hereafter,  ``\textit{square(s)}" refers to ``square cell(s)", and ``\textit{triangle(s)}" refers to ``triangular cells".


The goal of this paper is to prove Theorem \ref{gendoug} by using a bijection between domino tilings and non-intersecting lattice paths.

\section{Structure of generalized Douglas regions}

Our goal of this section is to investigate further the structure of generalized Douglas regions.

\begin{figure}\centering
\resizebox{!}{7cm}{
\begin{picture}(0,0)%
\includegraphics{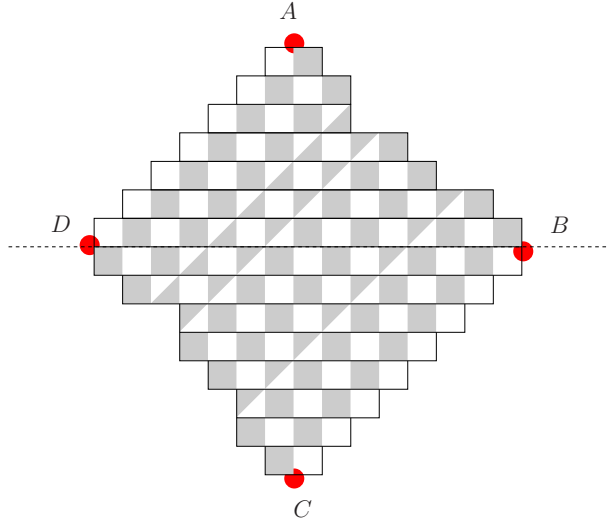}%
\end{picture}%
\setlength{\unitlength}{3947sp}%
\begingroup\makeatletter\ifx\SetFigFont\undefined%
\gdef\SetFigFont#1#2#3#4#5{\reset@font\fontsize{#1}{#2pt}
  \fontfamily{#3}\fontseries{#4}\fontshape{#5}
  \selectfont}
\fi\endgroup
\begin{picture}(4985,4362)(225,-3779)
\put(2481,367){\makebox(0,0)[lb]{\smash{{\SetFigFont{12}{14.4}{\rmdefault}{\mddefault}{\updefault}{$A$}}}}}
\put(2597,-3764){\makebox(0,0)[lb]{\smash{{\SetFigFont{12}{14.4}{\rmdefault}{\mddefault}{\updefault}{$C$}}}}}
\put(589,-1402){\makebox(0,0)[lb]{\smash{{\SetFigFont{12}{14.4}{\rmdefault}{\mddefault}{\updefault}{$D$}}}}}
\put(4725,-1405){\makebox(0,0)[lb]{\smash{{\SetFigFont{12}{14.4}{\rmdefault}{\mddefault}{\updefault}{$B$}}}}}
\end{picture}}
\caption{Partitioning a generalized Douglas region into horizontal strips of cells.}
\label{strip}
\end{figure}

Consider a generalized Douglas region $D_{a}(d_1,\dotsc,d_k)$. Denote by $p$  the number of rows of black square cells, denote by $q$ the number of rows of black up-pointing triangular cells, and denote by $l$ the number of rows of
black down-pointing triangular cells in the region.

The region $D_{a}(d_1,\dotsc,d_k)$ can be partitioned into $a$ horizontal strips of cells above $BD$ and $w$ horizontal strips of cells below $BD$ (see Figure \ref{strip} for an example with $a=7$, $k=4$, $d_1=4$, $d_2=2$, $d_3=5$, $d_4=4$). Consider the horizontal strips above segment $BD$. Each of them starts by a white square in the top row of cells, and ends by a black square or a black down pointing-triangle along the northeastern boundary of the region. Compare the number of starting cells and the number of ending cells in those strips, we get
 \begin{equation}\label{constraint1}
 a=p+l.
 \end{equation}

We consider now the horizontal strips below the segment $BD$. Each of them starts by a black square or a black up-pointing triangle along the southwestern boundary, and ends by a white square in the bottom rows of cells. Again, we compare of the number of starting cells and the number of ending cells in those strips, and obtain
 \begin{equation}\label{constraint2}
w=p+q.
 \end{equation}
From (\ref{constraint1}) and (\ref{constraint2}), we get
\begin{subequations}
\begin{equation}\label{constraint1'}
a+q-l=p+q=w,
\end{equation}
 \begin{equation}\label{constraint3}
 a+w=2p+q+l.
 \end{equation}
\end{subequations}

Consider the number of unit steps on the southwestern boundary of the region. Each row of black squares contributes 2 steps, and each row of black triangles contributes 1
 step to the latter number of steps. Thus, the number of steps here is exactly the expression on the right hand side of (\ref{constraint3}).
 On the other hand, one readily see that the number of steps on the southwestern boundary is equal to the sum of all distances $d_i$'s. Therefore,
\begin{equation}\label{constraint4}
 a+w=2p+q+l=\sum_{i=1}^k d_i.
\end{equation}

For  each of $k-1$ drawn-in diagonals of $D_{a}(d_1,\dotsc,d_k)$, there is exactly one row of black up-pointing triangles or one row of black down-pointing triangles with bases resting on it.
This implies that the number of rows of black triangles is equal to $k-1$, i.e.
\begin{equation}\label{constraint5}
q+l=k-1.
\end{equation}

The $k-1$ drawn-in diagonals divide the region into $k$ parts called \textit{layers}. The first layer is the part above the top drawn-in diagonal, the last layer is the part below the bottom drawn-in diagonal, and the $i$-th layer (for $1<i<k$) is the part between the $(i-1)$-th and the $i$-th drawn-in diagonals.

\section{Schr\"{o}der paths with barriers}

A \textit{Schr\"{o}der path} is a path in the plane, starting and ending on the $x$-axis, never going below the $x$-axis, using $(1, 1)$, $(1,-1)$ and $(2,0)$ steps
 (i.e. (diagonally) up, (diagonally) down and  flat steps, respectively). Denote by $\U$, $\D$, and $\F$ the up, down and  flat steps, respectively.

A \textit{barrier} is a length-1 horizontal segment in the plane. A Schr\"{o}der path is said to be \textit{compatible} with a setup of barriers if it does not cross any barriers of the setup.

 Let $a_1,\dots,a_{m}$ be nonnegative integers so that $a_1<a_2<\dotsc<a_m$. We consider a setup of barriers as follows. For any $k\in\mathbb{Z}$ and $1\leq i\leq m$,
  we draw a barrier connecting two points $(-a_i+k,k+\frac{1}{2})$ and $(-a_i+k+1,k+\frac{1}{2})$ (i.e. all barriers appear along the lines $y=x+a_i$, for $i=1,2,\dotsc,m$).
   Denote by $\Ba(a_1,a_2,\dotsc,a_m)$ the resulting setup of barriers.
   A \textit{bad flat step} (with respect to the setup $\Ba(a_1,a_2,\dotsc,a_m)$)  of a Schr\"{o}der path  is a flat step from $(x,0)$ to $(x+2,0)$, where $x\notin\{-a_i-1: 1\leq i\leq m\}$.
   Figure \ref{schroderfig} illustrates an example of a Schr\"{o}der path compatible with the set of barriers $\Ba(2,5,8)$, and the path has a bad flat step from $(1,0)$ to $(3,0)$.

\begin{figure}\centering
\includegraphics[width=10cm]{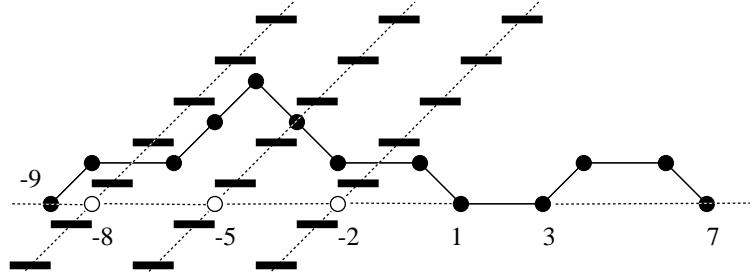}
\caption{A Schr\"{o}der path compatible with barrier set $\Ba(2,5,8)$.}
\label{schroderfig}
\end{figure}

Let $x_{i}$ be the $i$-th largest negative odd number in $\mathbb{Z}\backslash\{-a_1, \dotsc,-a_m\}$, let $A_{i}$ be the point $(x_i,0)$, and let $B_i$ be the point $(2i-1,0)$, for $i=1,2,\dotsc,n$.
 We consider two sets of $n$-tuples  of non-intersecting Schr\"{o}der paths compatible with $\Ba(a_1,a_2,\dotsc,a_m)$ as follows.

The set $\Pi_n(a_1,\dots,a_m)$ consists of $n$-tuples of non-intersecting Schr\"{o}der paths $(\pi_1,$ $\pi_2,$ $\dotsc,$ $\pi_n)$ (compatible with $\Ba(a_1,a_2,\dotsc,a_m)$),
 where $\pi_i$ connects two points $A_i$ and $B_i$. The set $\Lambda_n(a_1,\dots,a_m)$ consists of $n$-tuples of non-intersecting Schr\"{o}der paths $(\lambda_1,$ $\lambda_2,$ $\dotsc,$ $\lambda_n)$
  (compatible with $\Ba(a_1,a_2,\dotsc,a_m)$), where $\lambda_i$ connects $A_i$ and $B_i$, and has no bad flat steps.


Next, we quote a classic result about non-intersecting paths.

\begin{definition}
Let $G=(V,E)$ be an acyclic directed graph. If $I$ and $J$ are two ordered sets of vertices of $G$, then $I$ is said to be compatible with $J$ if,
whenever $u<u'$ in $I$ and $v>v'$ in $J$, every path $P\in \mathcal{P}(u,v)$ intersects every path $Q\in \mathcal{P}(u',v')$, where $\mathcal{P}(u,v)$
 (resp., $\mathcal{P}(u',v')$) is the set of paths in $G$ from $u$ to $v$ (resp., from $u'$ to $v'$).
\end{definition}
The following result is due to Lindstr\"{o}m-Gessel-Viennot (see \cite{GV}; \cite{Lind}, Lemma 1; \cite{Stem} Theorem 1.2)
\begin{lemma}\label{Lin}
Let $U=(u_1,u_2,\dotsc,u_n)$ and $V=(v_1,v_2,\dotsc,v_n)$ be two $n$-tuples of vertices in an acyclic digraph $G$.
If $U$ is compatible with $V$, then the number of $n$-tuples of non-intersecting paths connecting vertices in $U$ to vertices in $V$ is equal to $\det \left(\left(a_{i,j}\right)_{1\leq i,j\leq n}\right)$,
 where $a_{i,j}$ is the number of paths in $G$ from $u_i$ to $v_j$.
\end{lemma}

Given a setup of barriers $\Ba(a_1,\dots,a_m)$, where $a_1<a_2<\dotsc<a_m$. We define
\begin{equation}
H_n:=H_{n}(a_1,\dotsc,a_m):=\begin{bmatrix}
r_{1,1}&r_{1,2}&\dotsc&r_{1,n}\\
r_{2,1}&r_{2,2}&\dotsc&r_{2,n}\\
\vdots&\vdots&\ &\vdots\\
r_{n,1}&r_{n,2}&\dotsc&r_{n,n}
\end{bmatrix},
\end{equation}
 and
\begin{equation}
G_n:=G_{n}(a_1,\dotsc,a_m):=\begin{bmatrix}
s_{1,1}&s_{1,2}&\dotsc&s_{1,n}\\
s_{2,1}&s_{2,2}&\dotsc&s_{2,n}\\
\vdots&\vdots&\ &\vdots\\
s_{n,1}&s_{n,2}&\dotsc&s_{n,n}
\end{bmatrix},
\end{equation}
where $r_{i,j}$ (resp., $s_{i,j}$) is the number of Schr\"{o}der paths (resp., Schr\"{o}der paths without bad flat steps) from $A_i$ to $B_j$, where $A_i=(x_i,0)$
 with $x_{i}$ is the $i$th largest negative odd integer in $\mathbb{Z}\backslash\{-a_1, \dotsc,-a_m\}$, and where $B_j=(2j-1,0)$, for $1\leq i,j\leq n$.

\begin{proposition}\label{prop2} For any positive integers $n$ and $m$, we have
\begin{equation}\label{prop2eq1}
|\Pi_n(a_1,\dotsc,a_m)|=\det(H_n(a_1,\dotsc,a_m)),
\end{equation}
\begin{equation}\label{prop2eq2}
|\Lambda_n(a_1,\dotsc,a_m)|=\det(G_n(a_1,\dotsc,a_m)).
\end{equation}
\end{proposition}

\begin{proof}
 Consider two sets of points: $\mathcal{A}=\{A_1,\dotsc,A_n\}$ and $\mathcal{B}=\{B_1,\dots,B_n\}$, where $A_i$'s and  $B_j$'s are defined as in the paragraph before the statement of the theorem.

\begin{figure}\centering
\includegraphics[width=12cm]{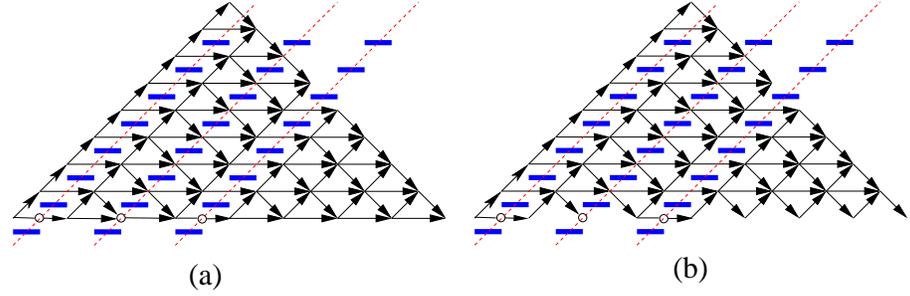}
\caption{Illustrating the proof of Proposition \ref{prop2}.}
\label{Graphschroder}
\end{figure}

Let $G$ be the digraph defined as follows. The vertex set of $G$ consists of all lattice points of the square lattice that are inside or on the edges of the up-pointing
 isosceles right triangle whose hypothenuse is segment $A_nB_n$, and that can be reached from $A_n$ by $(1,1)$, $(1,-1)$ and $(2,0)$ steps.
 An edge of $G$ connects from $(x,y)$ to $(x',y')$  if we can go from the former vertex to the latter vertex
by one of the above steps. Next, we remove all edges which cross some barriers of $\Ba(a_1,\dots,a_m)$
 (see the illustrative picture in Figure \ref{Graphschroder}(a), for $m=3$, $n=4$, $a_1=2$, $a_2=5$, and $a_3=8$).

Apply the Lemma \ref{Lin} to the digraph $G$ with two compatible sets of vertices $\mathcal{A}$ and $\mathcal{B}$, we get
\begin{equation}\label{GVeq}
\det (H_n)=|\Pi_n(a_1,\dotsc,a_m)|,
\end{equation}
which proves (\ref{prop2eq1}).

The equality (\ref{prop2eq2}) can be proved similarly. We apply the same procedure as in the proof of (\ref{prop2eq1}) to the digraph $G'$ that is obtained from the graph $G$
 by removing all horizontal edges on $x$-axis containing no points $(-a_i,0)$, for $1\leq i\leq m$ (see Figure \ref{Graphschroder}(b)).
\end{proof}

\begin{figure}\centering
\includegraphics[width=10cm]{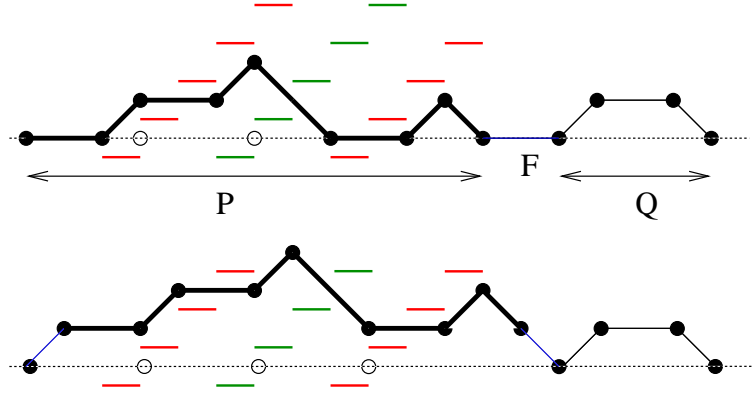}
\caption{A bijection between two types of Schr\"{o}der paths with barriers.}
\label{schrofig2}
\end{figure}

Similar to the relationship between large and small Schr\"{o}der numbers (see \cite{Eu} and \cite{Gessel}), we have the following fact about $r_{i,j}$ and $s_{i,j}$.
\begin{proposition}\label{prop3}
Given a setup of barriers $\Ba(a_1,\dotsc,a_m)$. If $a_1\not=0$, then $r_{i,j}=2s_{i,j}$, for any $1\leq i,j\leq n$.
\end{proposition}

\begin{proof}
It is easy to see $r_{1,1}=2=2s_{1,1}$. Thus, we assume in the rest of the proof that $i+j\geq 3$.

Fix two indices $i$ and $j$, so that $i+j\geq 3$. We consider the following two subsets of the set of all Schr\"{o}der paths from $A_i$ to $B_j$,
which are compatible with $\Ba(a_1,\dotsc,a_m)$:
\begin{enumerate}
\item[(i)] The set $S$ of the paths having at least one bad flat
step;
 \item[(ii)] The set $S'$ of the paths  having no bad flat steps.
 \end{enumerate}
 We have a bijection between $S$ and $S'$ working as follows.

Let $\tau$ be a Schr\"{o}der path in  $S$. We can factor $\tau=P\F Q$, where $\F$ is the
last bad flat step in $\tau$, so $Q$ has no bad flat steps (see the upper picture in Figure \ref{schrofig2}).
We define a Schr\"{o}der path $\lambda= \U P\D Q$ (see the lower picture in Figure \ref{schrofig2}).
 One readily sees that  $\lambda$ is compatible with the  setup of barriers $\Ba(a_1,\dotsc,a_m)$, and has no bad flat steps.
 It means $\lambda \in S'$. Since $\lambda$ is determined uniquely by $\tau$, this gives an injection from $S$ to $S'$.

On the other hand, let $\lambda$ be a Schr\"{o}der path in $S'$, and let $A^*=(c,0)$ the first returning point of $\lambda$ to $x$-axis.
We factor $\lambda=\overline{\lambda} Q$, where $\overline{\lambda}$ is the portion of $\lambda$ connecting $A_i$ and $A^*$. 
We can factor further $\overline{\lambda}=\U P \D$ by the choice of $A^*$. Next, we define a Schr\"{o}der path $\tau=P \F Q$.
We have the number $c-2$ is not in the set$\{-a_i-1: 1\leq i\leq m\}$ (otherwise the last step of $\overline{\lambda}$, which is a down step,
crosses a barrier, a contradiction). Thus, by definition, the flat step $\F$, from $(c-2,0)$ to $(c,0)$, in the factorization of $\tau$ is a bad flat step.
Moreover, $\tau$ is compatible with $\Ba(a_1,\dotsc,a_m)$, so $\tau \in S$.
 Since $\tau$ is determined uniquely by $\lambda$, this yields an injection from $S'$ to $S$.

Therefore, we have a bijection between $S$ and $S'$, which completes the proof of the lemma.
\end{proof}

\begin{proposition}\label{prop4}
For any positive integers $n,m$, and for any nonnegative integers  $a_1,$ $a_2,$ $\dotsc,$ $a_m$ so that $a_1<a_2-1<\dotsc<a_m-1$

(a) $ |\Pi_n(a_1,\dotsc,a_m)|=|\Lambda_{n+1}(a_1+2,\dotsc,a_m+2)|.$

(b) $|\Lambda_{n}(1,a_2\dotsc,a_m)|=|\Pi_{n}(a_2-2,\dotsc,a_m-2)|$  if $a_1=1$.

(c) $ |\Pi_n(0,a_2\dotsc,a_m)|=|\Pi_{n-1}(a_2-2\dotsc,a_m-2)|$ if $a_1=0$.

\end{proposition}
\begin{proof}
(a) We have a bijection $\varphi$ between two sets $\Pi_n(a_1,\dotsc,a_m)$ and $\Lambda_{n+1}(a_1+2,\dotsc,a_m+2)$
defined as follows. $\varphi$ carries $(\pi_1,\dotsc,\pi_n)\in \Pi_{n}(a_1,\dotsc,a_m)$ into $(\lambda_1,\dotsc,\lambda_{n+1})\in\Lambda_{n+1}(a_1+2,\dotsc,a_m+2)$,
where $\lambda_1=\U\D$ and $\lambda_{i+1}=\U^{(i-1)}\pi_i\D^{(i-1)}$, for $1\leq i\leq n$. This bijection is illustrated in Figure \ref{bijection3}.

\begin{figure}\centering
\resizebox{!}{3cm}{
\begin{picture}(0,0)%
\includegraphics{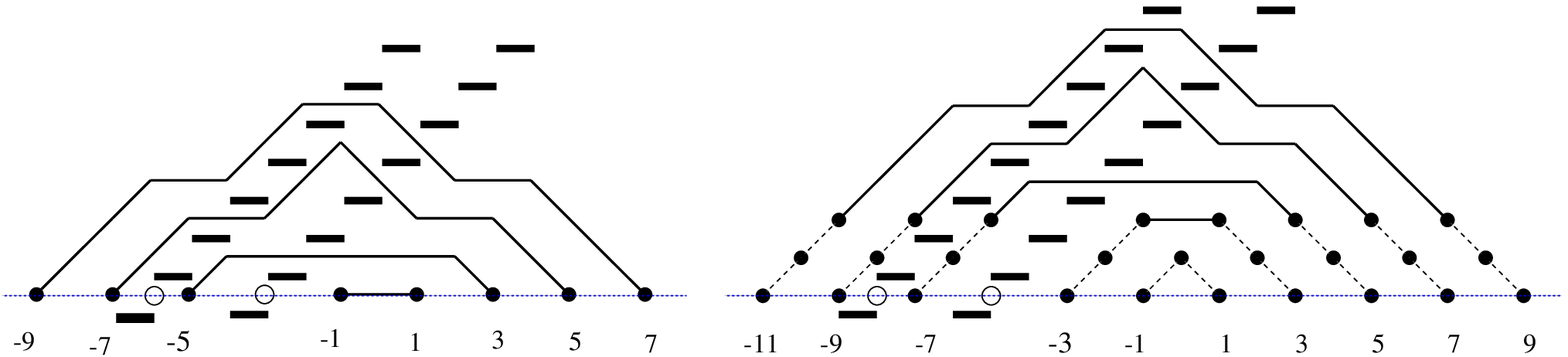}%
\end{picture}%
\setlength{\unitlength}{3947sp}%
\begingroup\makeatletter\ifx\SetFigFont\undefined%
\gdef\SetFigFont#1#2#3#4#5{%
  \reset@font\fontsize{#1}{#2pt}%
  \fontfamily{#3}\fontseries{#4}\fontshape{#5}%
  \selectfont}%
\fi\endgroup%
\begin{picture}(9710,2197)(2115,-2849)
\put(4768,-2594){\makebox(0,0)[lb]{\smash{{\SetFigFont{12}{12.0}{\rmdefault}{\mddefault}{\updefault}{$\pi_1$}%
}}}}
\put(5135,-2294){\makebox(0,0)[lb]{\smash{{\SetFigFont{12}{12.0}{\rmdefault}{\mddefault}{\updefault}{$\pi_2$}%
}}}}
\put(5353,-2092){\makebox(0,0)[lb]{\smash{{\SetFigFont{12}{12.0}{\rmdefault}{\mddefault}{\updefault}{$\pi_3$}%
}}}}
\put(5593,-1799){\makebox(0,0)[lb]{\smash{{\SetFigFont{12}{12.0}{\rmdefault}{\mddefault}{\updefault}{$\pi_4$}%
}}}}
\put(9164,-2266){\makebox(0,0)[lb]{\smash{{\SetFigFont{12}{12.0}{\rmdefault}{\mddefault}{\updefault}{$\lambda_1$}%
}}}}
\put(9804,-2055){\makebox(0,0)[lb]{\smash{{\SetFigFont{12}{12.0}{\rmdefault}{\mddefault}{\updefault}{$\lambda_2$}%
}}}}
\put(10040,-1819){\makebox(0,0)[lb]{\smash{{\SetFigFont{12}{12.0}{\rmdefault}{\mddefault}{\updefault}{$\lambda_3$}%
}}}}
\put(10335,-1583){\makebox(0,0)[lb]{\smash{{\SetFigFont{12}{12.0}{\rmdefault}{\mddefault}{\updefault}{$\lambda_4$}%
}}}}
\put(10630,-1406){\makebox(0,0)[lb]{\smash{{\SetFigFont{12}{12.0}{\rmdefault}{\mddefault}{\updefault}{$\lambda_5$}%
}}}}
\end{picture}}
\caption{Illustrating the proof of Lemma \ref{prop4}(a), for  $m=2$, $n=4$, $a_1=3$, and $a_2=6$.}
\label{bijection3}
\end{figure}

\medskip

(b) There is also a bijection $\psi$ between $\Pi_{n}(a_2-2\dotsc,a_m-2)$ and $\Lambda_{n}(1,a_2\dotsc,a_m)$  by  setting
\[\psi((\pi_1,\dotsc,\pi_n))=(\lambda_1,\dotsc,\lambda_{n})\in\Lambda_{n}(1,a_1,\dotsc,a_m),\]
where $\lambda_i:=\U\pi_i\D$, for $1\leq i\leq n$. This bijection is illustrated in Figure \ref{bijection7}.

\begin{figure}\centering
\includegraphics[width=13cm]{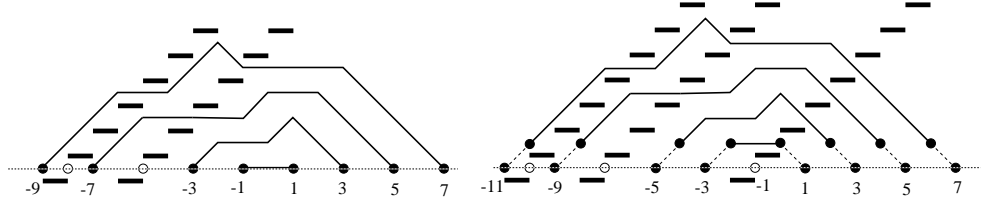}
\caption{Illustrating the proof of Lemma \ref{prop4}(b), for  $m=3$, $n=4$, $a_1=1$,  $a_2=7$ and $a_3=10$.}
\label{bijection7}
\end{figure}
\medskip

(c) We construct a bijection $\phi$ between two sets $\Pi_n(0,a_2\dotsc,a_m)$ and $\Pi_{n-1}(a_2-2\dotsc,a_m-2)$, for $n\geq2$, as follows.
 Let $(\pi_1,\dotsc,\pi_{n-1})$ be an element of $\Pi_{n-1}(a_2-2\dotsc,a_m-2)$.

It is easy to see that the last $i-1$ steps of $\pi_i$ are down steps, for $2\leq i\leq n-1$.
Thus, we can factor $\pi_i:=\widetilde{\pi}_i\D^{(i-1)}$, for $2\leq i\leq n-1$. Let $\pi'_1:=\F$, $\pi'_2:=\U\pi_1\F\D$ and $\pi'_i:=\U\widetilde{\pi}_{i-1}\F\D^{(i-1)}$, for $2\leq i\leq n$
 (see Figure \ref{bijection6}).  Define $\phi$ by setting
 \[\phi((\pi_1,\dotsc,\pi_{n-1})):=(\pi'_1,\dotsc, \pi'_n).\]
\end{proof}

\begin{figure}\centering
\includegraphics[width=13cm]{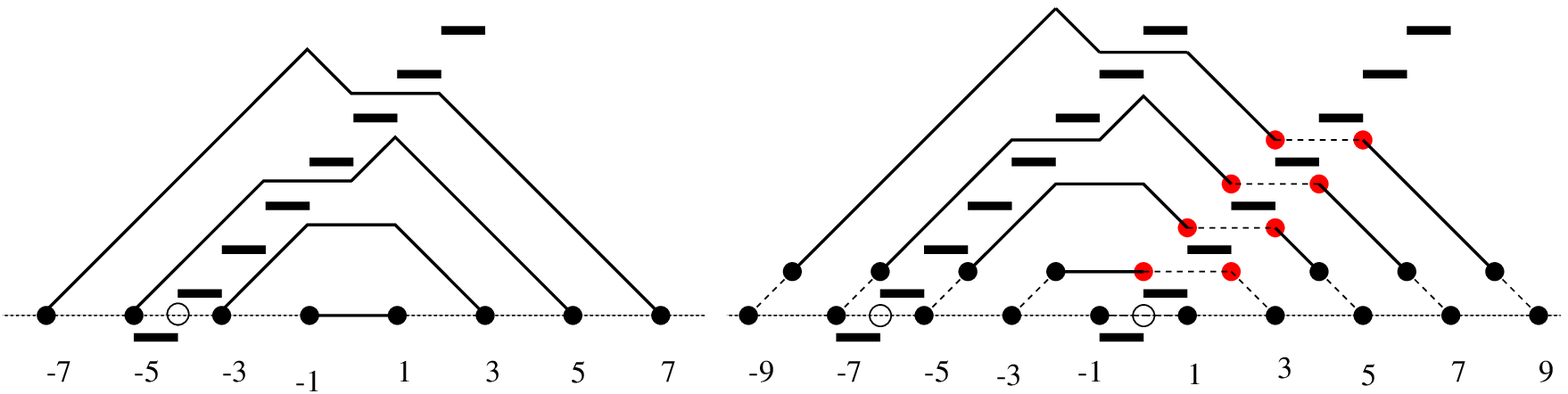}
\caption{Illustrating the proof of Lemma \ref{prop4}(c), for  $m=2$, $n=5$, $a_1=0$, and $a_2=6$.}
\label{bijection6}
\end{figure}

\section{Proof of Theorem \ref{gendoug}}

Before presenting the proof of Theorem \ref{gendoug}, we prove an important fact stated in the next proposition.

\begin{proposition}\label{prop1}
Assume that $a,d_1,\dotsc,d_k$ are positive integers so that the generalized Douglas region $D_{a}(d_1,d_2,\dotsc,d_k)$ has the width $w$, and has
 its western and eastern vertices on the same horizontal line. Let $a_i:=d_k+\dotsc+d_{k-i+1}+i-1$, for $i=1,2,\dotsc,k-1$.
Then
\begin{equation}\label{prop1eq}
\M(D_{a}(d_1,\dotsc,d_k))=|\Pi_w(a_1,\dotsc,a_{k-1})|.
\end{equation}
\end{proposition}

\begin{proof}
Consider a new region $\overline{\mathcal{D}}:=\overline{D}_{a}(d_1,\dotsc,d_k)$ associating with certain barriers as follows.
 We first deform the dual graph $G$ of $D_{a}(d_1,\dotsc,d_k)$ into a subgraph $G'$ of the infinite square grid $\mathbb{Z}^2$ (see Figure \ref{deform2}(a) for an example).
  Denote by $G''$ the subgraph of the square grid induced by the vertex set of $G'$. The region $\overline{\mathcal{D}}:=\overline{D}_{a}(d_1,\dotsc,d_k)$
   is the region in the square lattice having the dual graph (isomorphic to) $G''$. Consider the southwest-to-northeast lines passing the centers of a length-3
    horizontal step on the southwestern boundary or on the northeastern boundary of $\overline{\mathcal{D}}$.
    Draw the barriers at the positions of the horizontal lattice segments passed through by those lines (see Figure \ref{deform2}(b) for an example; the bold horizontal segments indicate the barriers).

\begin{figure}\centering
\includegraphics[width=12cm]{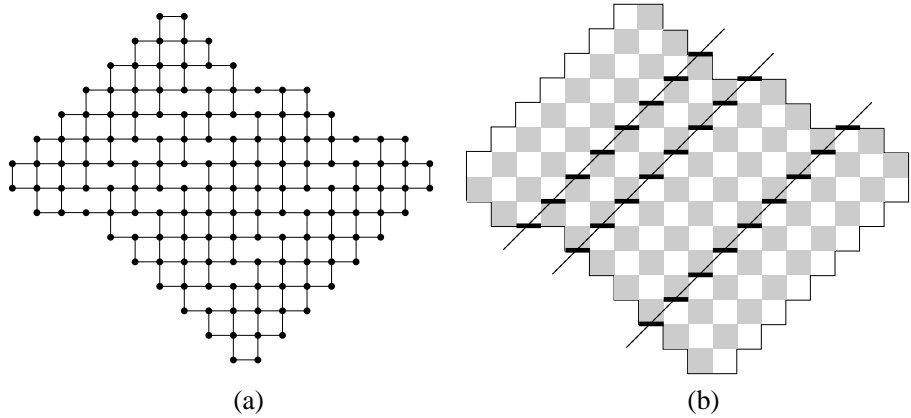}
\caption{The dual graph of $D_{7}(4,2,5,4)$ deformed into a subgraph of the square lattice (left), and region $\overline{D}_{7}(4,2,5,4)$ (right).}
\label{deform2}
\end{figure}

 A \textit{bad tile} of $\overline{\mathcal{D}}$ is a vertical domino whose center passed through by a barrier. A
  compatible tiling of $\overline{\mathcal{D}}$ is a tiling of $\overline{\mathcal{D}}$ which contains no bad tiles.  We have
\begin{equation}\label{deform}
\M(G')=\M^{*}(\overline{\mathcal{D}}),
\end{equation}
where $\M^{*}(\overline{\mathcal{D}})$ is the number of compatible tilings of $\overline{\mathcal{D}}$.
Indeed, the expression on the right of (\ref{deform}) is exactly the number of perfect matchings of the graph obtained from the dual graph $G''$ of $\overline{\mathcal{D}}$
by removing all the vertical edges corresponding to its bad tiles, i.e. the graph $G'$.

We have a bijection between the set of compatible tilings of $\overline{\mathcal{D}}$ and the set of $w$-tuples of non-intersecting
Schr\"{o}der paths $(\tau_1,\dotsc,\tau_w)$ compatible with the barriers of  $\overline{\mathcal{D}}$, where $\tau_i$ starts by the center of the $i$th vertical step (from bottom)
on the southwestern boundary, and ends by  the center of the $i$th vertical step on the southeastern boundary of $\overline{\mathcal{D}}$ (illustrated in Figure \ref{bijection2'}).  In particular the
bijection works as in the next paragraph.

It is easy to see that each tiling of $\overline{\mathcal{D}}$ gives a unique $w$-tuple of non-intersecting paths $(\tau_1,\dotsc,\tau_w)$.
 On the other hand, given a $w$-tuple of non-intersecting paths $(\tau_1,\dotsc,\tau_w)$, we can recover the corresponding tiling of the region as follows.
 The up and down steps in each path $\tau_i$ are covered by vertical dominos, and the flat steps are covered by horizontal dominos.
 After covering all steps of all paths $\tau_i$'s, we cover the rest of the region by horizontal dominos.

\begin{figure}\centering
\begin{picture}(0,0)%
\includegraphics{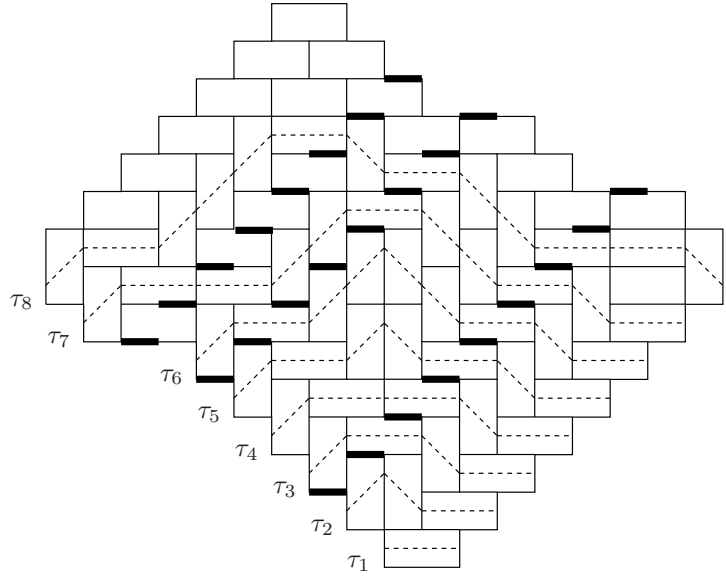}%
\end{picture}%
\setlength{\unitlength}{3947sp}%
\begingroup\makeatletter\ifx\SetFigFont\undefined%
\gdef\SetFigFont#1#2#3#4#5{%
  \reset@font\fontsize{#1}{#2pt}%
  \fontfamily{#3}\fontseries{#4}\fontshape{#5}%
  \selectfont}%
\fi\endgroup%
\begin{picture}(4507,3570)(222,-2948)
\put(2354,-2933){\makebox(0,0)[lb]{\smash{{\SetFigFont{10}{12.0}{\rmdefault}{\mddefault}{\updefault}{$\tau_1$}%
}}}}
\put(2127,-2704){\makebox(0,0)[lb]{\smash{{\SetFigFont{10}{12.0}{\rmdefault}{\mddefault}{\updefault}{$\tau_2$}%
}}}}
\put(1891,-2468){\makebox(0,0)[lb]{\smash{{\SetFigFont{10}{12.0}{\rmdefault}{\mddefault}{\updefault}{$\tau_3$}%
}}}}
\put(1654,-2232){\makebox(0,0)[lb]{\smash{{\SetFigFont{10}{12.0}{\rmdefault}{\mddefault}{\updefault}{$\tau_4$}%
}}}}
\put(1418,-1995){\makebox(0,0)[lb]{\smash{{\SetFigFont{10}{12.0}{\rmdefault}{\mddefault}{\updefault}{$\tau_5$}%
}}}}
\put(1182,-1759){\makebox(0,0)[lb]{\smash{{\SetFigFont{10}{12.0}{\rmdefault}{\mddefault}{\updefault}{$\tau_6$}%
}}}}
\put(473,-1523){\makebox(0,0)[lb]{\smash{{\SetFigFont{10}{12.0}{\rmdefault}{\mddefault}{\updefault}{$\tau_7$}%
}}}}
\put(237,-1287){\makebox(0,0)[lb]{\smash{{\SetFigFont{10}{12.0}{\rmdefault}{\mddefault}{\updefault}{$\tau_8$}%
}}}}
\end{picture}%
\caption{Bijection between domino tilings and non-intersecting paths}
\label{bijection2'}
\end{figure}

\begin{figure}\centering
\resizebox{!}{5cm}{
\begin{picture}(0,0)%
\includegraphics{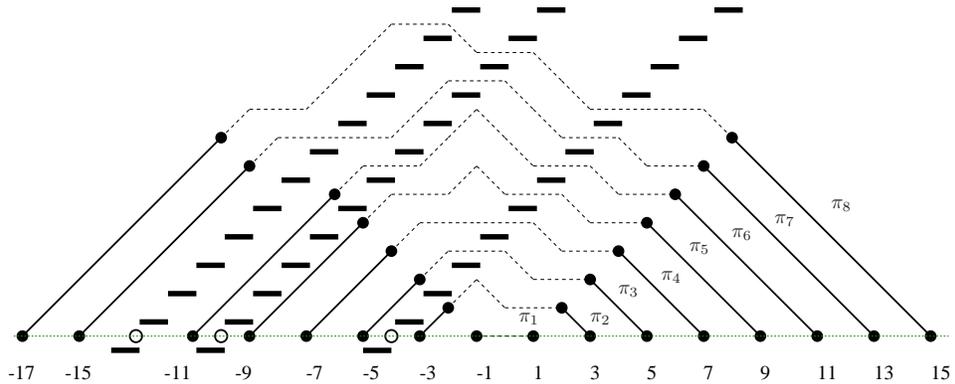}%
\end{picture}%
%
%
\setlength{\unitlength}{3947sp}%
\begingroup\makeatletter\ifx\SetFigFont\undefined%
\gdef\SetFigFont#1#2#3#4#5{%
  \reset@font\fontsize{#1}{#2pt}%
  \fontfamily{#3}\fontseries{#4}\fontshape{#5}%
  \selectfont}%
\fi\endgroup%
\begin{picture}(7899,3130)(428,-2956)
\put(4695,-2468){\makebox(0,0)[lb]{\smash{{\SetFigFont{10}{12.0}{\rmdefault}{\mddefault}{\updefault}{$\pi_1$}%
}}}}
\put(5286,-2468){\makebox(0,0)[lb]{\smash{{\SetFigFont{10}{12.0}{\rmdefault}{\mddefault}{\updefault}{$\pi_2$}%
}}}}
\put(5522,-2232){\makebox(0,0)[lb]{\smash{{\SetFigFont{10}{12.0}{\rmdefault}{\mddefault}{\updefault}{$\pi_3$}%
}}}}
\put(5876,-2114){\makebox(0,0)[lb]{\smash{{\SetFigFont{10}{12.0}{\rmdefault}{\mddefault}{\updefault}{$\pi_4$}%
}}}}
\put(6112,-1878){\makebox(0,0)[lb]{\smash{{\SetFigFont{10}{12.0}{\rmdefault}{\mddefault}{\updefault}{$\pi_5$}%
}}}}
\put(6467,-1760){\makebox(0,0)[lb]{\smash{{\SetFigFont{10}{12.0}{\rmdefault}{\mddefault}{\updefault}{$\pi_6$}%
}}}}
\put(6821,-1641){\makebox(0,0)[lb]{\smash{{\SetFigFont{10}{12.0}{\rmdefault}{\mddefault}{\updefault}{$\pi_7$}%
}}}}
\put(7293,-1523){\makebox(0,0)[lb]{\smash{{\SetFigFont{10}{12.0}{\rmdefault}{\mddefault}{\updefault}{$\pi_8$}%
}}}}
\end{picture}}
\caption{Bijection between $(\tau_1,\dotsc,\tau_8)$ in Figure \ref{bijection2'} and $(\pi_1,\dotsc,\pi_8)$}
\label{bijection3'}
\end{figure}

Next, we have a bijection between the set of $w$-tuples $(\tau_1,\dotsc,\tau_w)$ above and the set $\Pi_w(a_1,\dots,a_{k-1})$ of
$w$-tuples $(\pi_1,\dots,\pi_w)$ (shown in Figure \ref{bijection3'}).
Precisely, the Schr\"{o}der path $\pi_i$ is obtained from $\tau_i$ by adding $i-1$ up steps before its starting point, and adding $i-1$ down steps after its ending point,
i.e. $\pi_i:=\U^{(i-1)}\tau_i\D^{(i-1)}$, for $i=1,2,\dotsc,w$.

By the two above bijections and (\ref{deform}), we get (\ref{prop1eq}).
\end{proof}

We are now  ready to prove Theorem \ref{gendoug}.

\begin{proof}[\textbf{Proof of Theorem \ref{gendoug}}]
We prove (\ref{gendougeq}) by induction on the number of layers $k$ of the region.

For $k=1$, the region $D_{a}(d_1,\dotsc,d_k)$ is the Aztec diamond of order $a$, so (\ref{gendougeq}) follows from the Aztec diamond theorem \ref{Aztecthm}.

For the induction step, suppose (\ref{gendougeq}) holds for any generalized Douglas regions with strictly less than $k$ layers, for $k\geq 2$.
 We need to show that (\ref{gendougeq}) holds for any generalized Douglas region
$D_{a}(d_1,\dotsc,d_k)$.

Let  $a_i:=d_k+\dotsc+d_{k-i+1}+i-1$, for $i=1,2,\dotsc,k-1$, as in Proposition \ref{prop1}.
Recall that we denote by $p,q,l$ the numbers of rows of black squares, of black up-pointing triangles, and of black down-pointing triangles, respectively.

 There are two cases to distinguish, based on the parity of $d_k$.

\medskip

\quad\textit{ Case I. $d_k$ is even.}

Assume that $d_k=a_1=2x$, for some $x\geq 1$. The last layer of the region has $x$ rows of black square, so $p\geq x$; and the $(k-1)$th layer has a row of black up-pointing triangles with bases resting on
the last drawn-in diagonal, so $q\geq 1$. Thus,
$w=p+q\geq x+1$ (by (\ref{constraint2})).

By Proposition \ref{prop1}, we have
 \begin{equation}\label{eq1}
 \M(D_{a}(d_1,\dotsc,d_k))=|\Pi_w(a_1,\dotsc,a_{k-1})|.
 \end{equation}
By Propositions \ref{prop2} and \ref{prop3}, we obtain
 \begin{equation}\label{eq2}
 |\Pi_w(a_1,\dotsc,a_{k-1})|=\det(H_w)=2^w\det(G_w)=2^w|\Lambda_w(a_1,\dotsc,a_{k-1})|.
 \end{equation}
 We apply Proposition \ref{prop4}(a), and obtain
 \begin{equation}\label{eq3}
 |\Lambda_w(a_1,\dotsc,a_{k-1})|=|\Pi_{w-1}(a_1-2,\dotsc,a_{k-1}-2)|.
 \end{equation}
 Two equalities (\ref{eq2}) and (\ref{eq3}) imply
  \begin{equation}\label{eq4}
 |\Pi_w(a_1,\dotsc,a_{k-1})|=2^w|\Pi_{w-1}(a_1-2,\dotsc,a_{k-1}-2)|.
 \end{equation}
We apply (\ref{eq4})  $x$ times, obtain
   \begin{equation}\label{eq5}
 |\Pi_w(a_1,\dotsc,a_{k-1})|=2^{\sum_{i=0}^{x-1}(w-i)}|\Pi_{w-x}(0,a_2-a_1\dotsc,a_{k-1}-a_1)|.
 \end{equation}

By equality (\ref{constraint1}),  we have $a=p+l\geq p\geq x$. There are now two subcases to distinguish, depending on the value of $a$.

\medskip
\quad\textit{Case I.1. $a=x$.}

The equality (\ref{constraint1}) implies $p=x$ and $l=0$.  By (\ref{constraint5}), we have $q=k-1$; and by (\ref{constraint4}), we obtain
\[a+w=2p+n+m=2x+k-1=\sum_{i=1}^kd_i.\]
 Since $d_{k}=2x$, we have $d_1=d_2=\dotsc=d_{k-1}=1$ (see Figure \ref{schrodercase1} for an example of the generalized Douglas region in this case).
 Moreover, by (\ref{constraint2}),  we get $w=p+q=x+k-1$. It is easy to see that
\[|\Pi_{w-x}(0,a_2-a_1\dotsc,a_{k-1}-a_1)|=|\Pi_{k-1}(0,2,4,\dotsc,2(k-1))|=1,\]
so
\[\M(D_{a}(d_1,\dotsc,d_k))=2^{\sum_{i=0}^{x-1}(w-i)}.\]
One can verify that $\mathcal{C}=(w+1)x+\sum_{i=0}^{w-a-1}(w-i)$, then (\ref{gendougeq}) follows.

\begin{figure}\centering
\includegraphics[width=6cm]{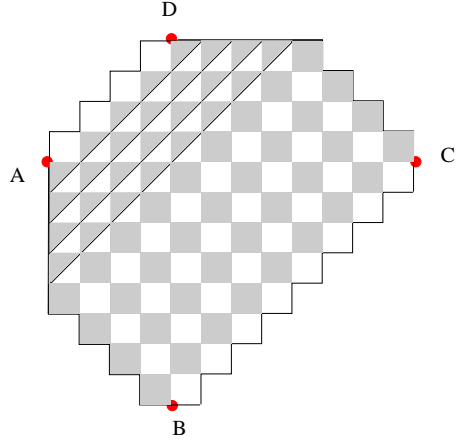}
\caption{Illustrating the region in Case I.1.}
\label{schrodercase1}
\end{figure}

\medskip

\quad\textit{Case I.2. $a>x$.}

\medskip
Then there is some $d_i>1$, for some $1\leq i\leq k-1$, so the $i$th layer has at least one row of black squares.
Since the last layer still have $x$ rows of black squares, we have $p\geq x+1$.  Since we already have $q\geq 1$ from the argument at the begining of Case I,  $w=p+q\geq x+2$.
By Proposition \ref{prop4}(c), we get
   \begin{equation}\label{eq6}
 |\Pi_w(a_1,\dotsc,a_{k-1})|=2^{\sum_{i=1}^{x-1}(w-i)}|\Pi_{w-x-1}(a_2-a_1-2\dotsc,a_{k-1}-a_1-2)|.
 \end{equation}

\begin{figure}\centering
\includegraphics[width=7cm]{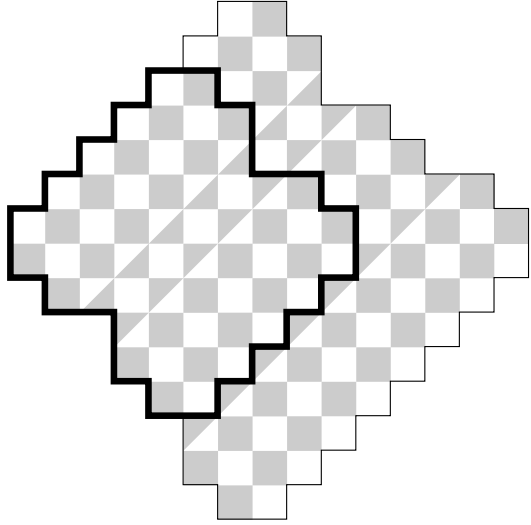}
\caption{Comparison of two regions $\mathcal{D}=D_{7}(4,2,5,4)$ and $\mathcal{D}'=D_{5}(4,2,4)$ (restricted by the bold contour).}
\label{compare3}
\end{figure}

Consider a new generalized Douglas region $\mathcal{D}':=D_{a-x}(d_1,\dotsc,d_{k-1}-1)$ having $k-1$ layers. Assume that $\mathcal{C}'$ is the
 number of  black regular cells in $\mathcal{D}'$, and $w'$ is the width of $\mathcal{D}'$. Intuitively, $\mathcal{D}'$ is obtained from $\mathcal{D}$ by removing its last layer and the row of
 black up-pointing triangles right above the last layer, and reducing the length of all remaining rows
of cells by $x$ units (see Figure \ref{compare3} for an example). Therefore, one can see that  $w-w'=x+1$, and $\mathcal{C}-\mathcal{C}'=(w+1)x+w+x(w-x-1)$. Thus, by induction hypothesis
\begin{equation}\label{eq9}
\M(\mathcal{D}')=2^{\mathcal{C}'-w'(w'+1)/2}=2^{\mathcal{C}-(w+1)x-w-x(w-x-1)-(w-x-1)(w-x)/2}.
\end{equation}
Moreover, by Proposition \ref{prop1}, we get
 \begin{equation}\label{eq7}
 \M(\mathcal{D}')=|\Pi_{w-x-1}(a_2-a_1-2\dotsc,a_{k-1}-a_1-2)|.
 \end{equation}
By (\ref{eq1}), (\ref{eq6}), (\ref{eq9}) and (\ref{eq7}), we obtain
 \begin{align}\label{eq8}
\M(D_{a}(d_1,\dotsc,d_k))&=2^{\sum_{i=0}^{x-1}(w-i)}\M(\mathcal{D}')\notag\\
&=2^{\sum_{i=0}^{x-1}(w-i)}2^{\mathcal{C}-(w+1)x-w-x(w-x-1)-(w-x-1)(w-x)/2},
\end{align}
which implies (\ref{gendougeq}).

\medskip
\quad\textit{Case II. $d_k$ is odd.}

\medskip
 Assume that $d_k=a_1=2x+1$, for some $x\geq0$. By (\ref{constraint3}), (\ref{constraint4}), and (\ref{constraint5}), we have
\begin{equation}\label{constraint6}
2p+k-1=2p+m+n=\sum_{i=1}^{k-1}d_i+2x+1\geq 2x+1+k-1.
\end{equation}
Thus, $p\geq x+1$, and by (\ref{constraint2}), we imply $w=p+q\geq x+1$. Note
that the last layer has now one row of black down-pointing triangles right below the last drawn-in diagonal. Thus, $l\geq 1$, and by (\ref{constraint1}), $a=p+l\geq x+2$.

We have also the two equalities (\ref{eq1})  and (\ref{eq4}) as in Case 1. We apply (\ref{eq4}) $x$ times, and obtain
 \begin{equation}\label{eq5'}
 |\Pi_w(a_1,\dotsc,a_{k-1})|=2^{\sum_{i=0}^{x-1}(w-i)}|\Pi_{w-x}(1,a_2-a_1+1\dotsc,a_{k-1}-a_1+1)|.
 \end{equation}
By Propositions \ref{prop2} and \ref{prop3}, we have
\begin{equation}\label{eq4'}
\M(\mathcal{D})=2^{\sum_{i=0}^{x-1}(w-i)}2^{w-x}|\Lambda_{w-x}(1,a_2-a_1+1\dotsc,a_{k-1}-a_1+1)|.
\end{equation}

There are also two subcases to distinguish, based on the value of $w$.

\medskip
\quad\textit{Case II.1. $w=x+1$.}

\medskip
By (\ref{constraint2}), we have $q=0$ and $p=x+1$. The equality (\ref{constraint6}) implies that $\sum_{i=1}^{k-1}d_i=k$.
Moreover, if $d_i=2$ for some $1<i\leq k-1$, then the $(i-1)$th layer has a row of up-pointing triangles with bases resting on the $i$th drawn-in diagonal,
 a contradiction to the fact that $q=0$. Therefore, we must have $d_1=2$ and $d_2=d_3=\dotsc=d_{k-1}=1$ (see Figure \ref{schrodercase2} for an example of this case).

We have now
\[|\Lambda_{w-x}(1,a_2-a_1+1\dotsc,a_{k-1}-a_1+1)|=|\Lambda_{1}(1,3,5,\dotsc,2k-3)|.\]
We can partition $\Lambda_{1}(1,3,5,\dotsc,2k-3)=\bigcup_{i=1}^kS_i$, where $S_k$ is the set of paths starting by exactly $i$ up steps. It is easy to see that $|S_i|$ is the number of lattice paths using
$(1,-1)$ and $(2,0)$ steps from $(-2k+1+i,i)$ to $(0,1)$ (see Figure \ref{schrodercase21} for an example with $k=6$; the black dots indicate the points $(-2k+1+i,i)$'s). Thus,  $|S_i|=\binom{i-1}{k-1}$, for $1\leq i\leq k$, and
\[|\Lambda_{w-x}(1,a_2-a_1+1\dotsc,a_{k-1}-a_1+1)|=|\Lambda_{1}(1,3,5,\dotsc,2k-3)|=\sum_{i=1}^{k}\binom{i-1}{k-1}=2^{k-1}.\]
By (\ref{eq4'}), we have $\M(\mathcal{D})=2^{k-1+\sum_{i=0}^{x}(w-i)}$. It is easy to see that $a=w+k-1$ and $\mathcal{C}=(w+1)x+a$, so (\ref{gendougeq}) follows.
\begin{figure}\centering
\includegraphics[width=6cm]{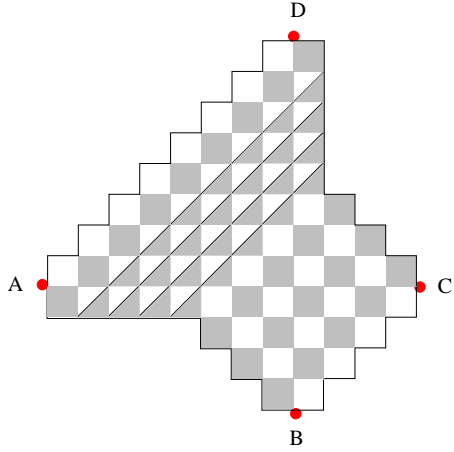}
\caption{Illustrating the region in Case II.1}
\label{schrodercase2}
\end{figure}

\begin{figure}\centering
\includegraphics[width=8cm]{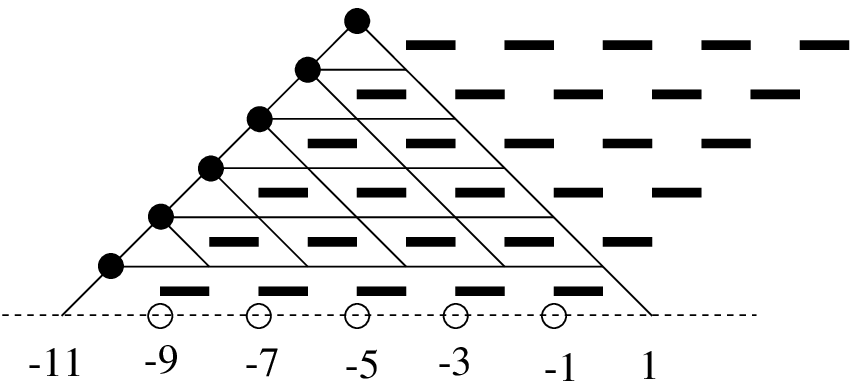}
\caption{}
\label{schrodercase21}
\end{figure}

\begin{figure}\centering
\includegraphics[width=10cm]{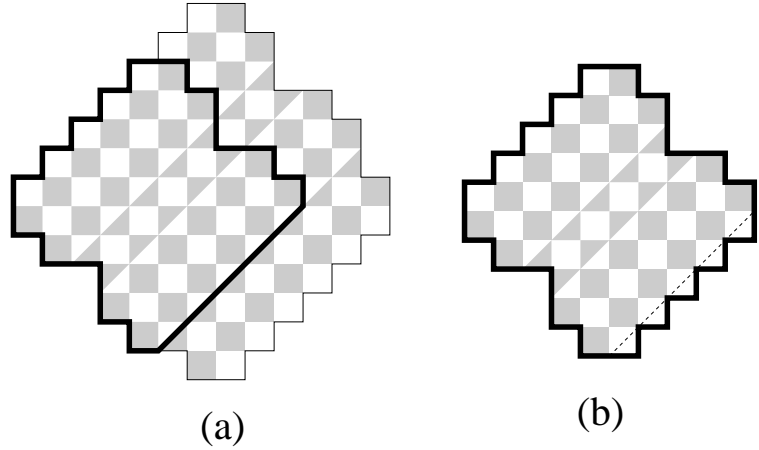}
\caption{Comparison of two regions $\mathcal{D}$ and $\mathcal{D}''$ (right).}
\label{compare4}
\end{figure}

\medskip
\quad\textit{Case II. 2. $w>x+1$.}

\medskip

By  (\ref{eq4'}) and Proposition \ref{prop4}(b), we obtain
  \begin{align}\label{eq6'}
 \M(\mathcal{D})&=2^{\sum_{i=1}^{x}(w-i)}|\Lambda_{w-x}(1,a_2-a_1+1\dotsc,a_{k-1}-a_1+1)|\notag\\
&=2^{\sum_{i=1}^{x}(w-i)} |\Pi_{w-x-1}(a_2-a_1-1\dotsc,a_{k-1}-a_1-1)|.
 \end{align}
Consider the region $\mathcal{D}'':=D_{a-x-1}(d_1,\dotsc,d_{k-1})$ (note that we already showed that $a\geq x+2$ at the begining of Case II). Assume that $\mathcal{C}''$ is the number of regular black
 cells in $\mathcal{D}''$, and $w''$ is the width of $\mathcal{D}''$. The region $\mathcal{D}''$ is obtained from $\mathcal{D}$
 by removing its last layer, reducing the length of all remaining rows of cells by $x+1$ units
  (see the region restricted by the bold contour in Figure \ref{compare4}(a)), and replacing the bottom row of white triangles in the resulting region by a row of white squares
   (see Figure \ref{compare4}(b)). Therefore, $w-w''=x$ and $\mathcal{C}-\mathcal{C}''=(w+1)x+(x+1)(w-x)$. By induction hypothesis, we have
\begin{equation}\label{eq9'}
\M(\mathcal{D}'')=2^{\mathcal{C}''-w''(w''+1)/2}=2^{\mathcal{C}-(w+1)x-(x+1)(w-x)-(w-x)(w-x+1)/2}.
\end{equation}
Moreover, by Proposition \ref{prop1}, we get
 \begin{equation}\label{eq7'}
 \M(\mathcal{D}'')=|\Pi_{w-x-1}(a_2-a_1-1\dotsc,a_{k-1}-a_1-1)|.
 \end{equation}
Finally, by (\ref{eq6'}),  (\ref{eq9'})  and (\ref{eq7'}), we obtain
 \begin{align}\label{eq8'}
 \M(D_{a}(d_1,\dotsc,d_k))&=2^{\sum_{i=0}^{x}(w-i)} \M(\mathcal{D}'')\notag\\
 &=2^{\sum_{i=0}^{x}(w-i)}2^{\mathcal{C}-(w+1)x-(x+1)(w-x)-(w-x)(w-x+1)/2},
 \end{align}
which implies (\ref{gendougeq}).
\end{proof}

\section{Concluding remarks}
We have two other proofs of the main theorem, Theorem \ref{gendoug}, using Kuo's graphical condensation \cite{Kuo} and a certain reduction theorem due to Propp \cite{propp}, respectively.
We present them in a subsequent paper \cite{Tri}.

\end{document}